\documentclass[16pt]{article}
\usepackage{graphicx,amssymb,amsmath,amsthm}
\usepackage{comment,cite,color}
\usepackage{cite,color}
\usepackage{algorithmic}
\usepackage{mathrsfs}
\usepackage[ruled]{algorithm}
\usepackage{epsfig}

\newtheorem{thm}{Theorem}[section]

\newtheorem{lem}[thm]{Lemma}

\theoremstyle{definition}
\newtheorem{defn}[thm]{Definition}
\newtheorem{example}[thm]{Example}
\newtheorem{remark}[thm]{Remark}

\newcommand{\R}{\mathbb{R}}
\newcommand{\C}{\mathbb{C}}
\newcommand{\Z}{\mathbb{Z}}
\newcommand{\N}{\mathbb{N}}

\date{}

\begin{document}
\bibliographystyle{plain}
\title{ Deterministic Sampling of Sparse Trigonometric Polynomials \thanks{
 Project supported by the Funds for Creative Research Groups of China (Grant No. 11021101)
and by NSFC grant 10871196 and National Basic Research Program of China (973 Program
2010CB832702) }}

\author{ \ Zhiqiang Xu}

\maketitle

\begin{abstract}
One can recover sparse multivariate trigonometric polynomials from
few randomly taken samples with high probability (as shown by Kunis and Rauhut). We give a deterministic sampling
of multivariate trigonometric polynomials inspired by Weil's exponential sum.  Our sampling can produce a deterministic matrix satisfying the statistical
 restricted isometry property, and also nearly optimal Grassmannian frames.
We show that  one can exactly reconstruct every $M$-sparse multivariate
trigonometric polynomial with fixed degree and of length $D$ from the
determinant sampling $X$, using the orthogonal matching pursuit, and $|X|$ is a
prime number greater  than $(M\log D)^2$. This result is optimal within the $(\log D)^2 $ factor. The simulations show that the
 deterministic sampling can offer reconstruction  performance similar to the random sampling.
\end{abstract}

\section{Introduction}

We investigate the problem of reconstructing sparse multivariate
trigonometric polynomials from few samples on $[0, 1]^{d}$. Let
$\Pi_q^d$ denote the space of all trigonometric polynomials of
maximal order $q\in \N$ in dimension $d$. An element  $f$ of
$\Pi_q^d$ is of the form
$$
f(x)\,\,=\!\!\sum_{k\in [-q,q]^d\cap \Z^d} c_k e^{2\pi ik\cdot x}, \quad x\in
[0,1]^d,
$$
where $c_k\in \C$. The dimension of $\Pi_q^d$ will be denoted by
$D:=(2q+1)^d$.
 We denote the support of the sequence of coefficients
of $c_k$ by $T$, i.e.,
$$
T:=\{k:c_k\neq 0\}.
$$
Throughout this paper, we will mainly deal with ``sparse"
trigonometric polynomials, i.e., we assume that  $|T|$ is much
smaller than the dimension of $D$ of $\Pi_q^d$. We set
$$
\Pi_q^d(M) := \bigcup_{T\subset [-q,q]^d\cap \Z^d\atop |T|\leq M} \Pi_T,
$$
where, $\Pi_T$ denotes the space of all trigonometric polynomials
whose coefficients are supported on $T$. Note that the set $\Pi_q^d(M)$  is the union of linear spaces and  consists of all trigonometric polynomials whose Fourier coefficients are supported
on a set $T\subset [-q,q]^d\cap \Z^d$ satisfying $|T|\leq M$.
 The aim of the paper is to sample
a trigonometric polynomial $f\in \Pi_q^d(M)$ at $N$ points and try to
reconstruct $f$ from these samples. We denote the sampling set
$$
X:=\{x_1,\ldots,x_N\}.
$$
We would like to reconstruct $f\in \Pi_q^d(M)$ from its sample
values
$$
y=f(x), \quad x\in X.
$$
 We use a  decoder $\Delta$ that maps
from $\C^N$ to $\Pi_q^d$, and the role of $\Delta$ is to provide an
approximation  to  $f$.
 The previous work concerns on the randomly taken samples. In \cite{tao1,tao2}, the authors choose $X$ by takeing  samples randomly on a lattice  and
use the {\em Basis Pursuit} (BP) as the decoder $\Delta$. In
\cite{trig1}, the result is generalized  for the case of
$x_1,\ldots,x_N$ being the uniform distribution on
$[0,1]^d$. We state the result as follows:
\begin{thm}\cite{trig1} Assume $f\in \Pi_q^d(M)$ for some
sparsity $M\in \N$. Let $x_1,\ldots,x_N\in [0,1]^d$ be
independent random variables having the uniform distribution on
$[0,1]^d$. If for some $\epsilon>0$ it holds
$$
N\geq CM\log(D/\epsilon)
$$
then with probability at least $1-\epsilon$ the trigonometric
polynomial $f$ can be recovered from its sample values $f(x_j)$,
$j=1,\ldots,N$, by Basis Pursuit. The constant $C$ is absolute.
\end{thm}

In \cite{trig2}, one chooses $\Delta$ as Orthogonal Matching Pursuit (Algorithm 1) and $X$ is
chosen according to one of two probability models:

(1) The sampling points $x_1,\ldots,x_N$ are independent random
variable having the uniform distribution on $[0,1]^d$;

(2) The sampling points $x_1,\ldots,x_N$ have the uniform
distribution on the finite set $2\pi\Z_m^d/m$ for some $m\in
\N\setminus \{1\}$.

\begin{thm}\label{th:trig2}\cite{trig2} Let $X=(x_1,\ldots,x_N)$ be chosen according to one of two probability models. Suppose that
$$
N\geq CM^2\log (D/\epsilon).
$$
 Then, with probability at least $1-\epsilon$, OMP recovers every $M$-sparse trigonometric polynomial. The constant $C$ is absolute.
\end{thm}

\begin{algorithm}
\begin{algorithmic}
 \STATE {\bf Input:}  sampling matrix ${\mathcal F}_X$, sampling vector $y=(f(x_j))_{j=1}^N$, maximum allowed sparsity $M$ or residula tolerance  $\epsilon$
 \STATE {\bf
Output:} the Fourier coefficients $c$ and its support $T$.
 \STATE {\bf Initialize:} $r^0=y,c^0=0,\Lambda^0=\emptyset, \ell=0$.
  \WHILE { $\|r^\ell\|_2>\epsilon$ or $\ell<M$}
  \STATE {\bf match:} $h^\ell={\mathcal F}_X^Tr^\ell$
  \STATE{\bf identity:} $\Lambda^{\ell+1}=\Lambda^\ell\cup\{{\rm argmax}_j|h^\ell(j)|\}$
  \STATE{\bf update:}  $c^{\ell+1}={\rm argmin}_{z: {\rm supp}(z)\subset \Lambda^{\ell+1}}\|y-{\mathcal F}_Xz\|_2$
  \STATE \,\,\qquad\qquad $r^{\ell+1}=y-{\mathcal F}_Xc^{\ell+1}$
  \STATE \,\,\qquad\qquad $\ell=\ell+1$
 \ENDWHILE
\end{algorithmic}
\caption{\small{Orthogonal Matching Pursuit}}
\end{algorithm}

The aim of this paper are twofold. First, we present  a deterministic sampling $X$   and
show   that OMP can recover every $f\in \Pi_q^d(M)$ exactly using
 $X$ with $|X|$ being a prime number greater  than $C(M\log D)^2$, provided $q$
 is fixed. So, we improve Theorem \ref{th:trig2}.
  Second, we construct a StRIP matrix with large range size.
 In particular, we exploit the connections between the exponential sum and the RIP matrix.

We now discuss the organization of this paper and we summarize its
main contributions. In Section 2, we introduce the deterministic
sampling and show that the coherence of the corresponding  sampling matrix less or
equal to  $ {(d-1)}/{\sqrt{N}} $ provided $N\geq 2q+1$
 and $N$ is a prime number. As a conclusion, OMP algorithm can  recover $f\in
\Pi_q^d(M)$ exactly when $N\geq  (d-1)^2 (2M-1)^2+1$ with the additional assumption of $N$ being prime number not less than $2q+1$. One has shown that if
$N\leq CM^{3/2}$, then with high probability there exists an
$M$-sparse coefficient vector $c$ depending on the sampling set such
that OMP fails (one also believes that the bound can be improved to
$O(M^2)$ ) (see \cite{OMP}). So, if one requires exactly recover of
{\em all} sparse trigonometric polynomials from a single sampling
set $X$, within the factor $\log^2 D$, our
deterministic sampling size almost meets the optimal bound for $f\in
\Pi_q^d(M)$ where  $q$ is fixed and  $d$ is variable. In Section 3, we show
that the $N\times D$ determining sampling matrix obeys the
Statistical Restricted Isometry Property (StRIP) of order $O(N (\log
(D/N))^2/(\log D(\log N)^2))$.  Though there are many deterministic
 sensing matrixes satisfying the StRIP \cite{quad,strip1},  the restriction to
 the size of our matrixes is light comparing with them.  Moreover,
 the StRIP matrix also implies nearly optimal Grassmannian frames. In Section 4, we  show that the
 deterministic sampling can provide similar reconstruction performance
 to that of the random ones by numerical experiments.

\section{Deterministic sampling}

\subsection{Exponential sums and sampling matrixes}
We first introduce a result in number theory, which plays a center role in our determinant sampling.
\begin{thm}\cite{weil}\label{th:weil} Suppose that $p$ is a prime number.
Suppose $f(x)=m_1x+\cdots+m_dx^d$ and there is a $j$, $1\leq j\leq d$,
so that $p\nmid m_j$. Then
$$
\left|\sum_{x=1}^p e^{\frac{2\pi i f(x)}{p}}\right|\,\, \leq\,\,
(d-1)\sqrt{p}.
$$
\end{thm}
We furthermore introduce the definition of {\em mutual incoherence}.
Let matrix $A=(a_1,\ldots,a_D)\in \C^{N\times D}$, where $N\leq D$
and $\|a_i\|_2=1$. The mutual incoherence of $A$ is defined by
$$
{\mathcal M}(A):=\max_{i\neq j}\left<a_i,a_j\right>.
$$
The low bound of ${\mathcal M}(A)$, which is also called as {\em
Welch's bound}, is given in \cite{welch}
\begin{equation}\label{eq:welch}
{\mathcal M}(A)\geq \sqrt{\frac{D-N}{(N-1)D}}.
\end{equation}
If the equality holds, we call $A$ as {\em optimal Grassmannian frames}. As shown in \cite{feichtinger},
 the equality can hold only if $D\leq N^2$ (see also \cite{mu}).
 Suppose that $N\geq 2q+1$ is a prime number.
We choose the sampling set $X=\{x_1,\ldots,x_N\}$ with
\begin{equation}\label{eq:x}
x_j= (j,j^2,\ldots,j^{d})/N \mod 1, \quad j=1,\ldots,N,
\end{equation}
 and denote by ${\mathcal F}_X$ the $N\times D$ {\em sampling matrix } with entries
$$
({\mathcal F}_{X})_{j,k} =\exp{(2\pi ik\cdot x_j)},\quad\quad\quad
j=1,\ldots,N,\,\, k\in [-q,q]^d.
$$
Also, $f(x_j)=({\mathcal F}_Xc)_j$, where $c$ is the
vector of Fourier coefficients of $f$. We let $\phi_k$ denote the
$k$-th column of ${\mathcal F}_X$. A simple observation is
$\|\phi_k\|_2=\sqrt{N}$. Set
$$\mu:={\mathcal M}({{\mathcal F}_X}/{\sqrt{N}}).$$
Then we have
\begin{lem}\label{le:co}
$$
\mu\,\,\leq\,\, {(d-1)}/{\sqrt{N}}.
$$
\end{lem}
\begin{proof}
Recall that $\phi_k$ denotes the $k$-th column of ${\mathcal F}_X$.
Note that
\begin{eqnarray*}
\left| \left<{\phi}_m,{\phi}_k\right>\right| = \left|\sum_{j=1}^N
e^{ 2\pi i p(j)/N}\right|,
\end{eqnarray*}
where $p$ is a polynomial with degree $d$ in the form of
$(m-k)\cdot (j,\ldots,j^{d})$ with $m,k\in [-q,q]^d$. Then Theorem \ref{th:weil} implies
that
$$
 \left| \left<{\phi}_m,{\phi}_k\right>\right|\,\, \leq \,\,
 (d-1)\sqrt{N}, \quad \mbox{\rm for   }\quad m\neq k.
 $$
Hence,
$$
\mu =\max_{m\neq k} \left<{\phi}_m,{\phi}_k\right>/N\,\,\leq\,\,
{(d-1)}/{\sqrt{N}}.
$$
\end{proof}
Let us consider recovery by OMP using the deterministic sampling $X$.
We first recall the following theorem:
\begin{thm}\cite{trig2}\label{th:omp} Assume $(2M-1)\mu <1$. Then OMP (and also
BP) recovers every $f\in \Pi_q^d(M)$.
\end{thm}
Combining Lemma \ref{le:co} and Theorem \ref{th:omp},  we have
\begin{thm}\label{th:main1}
Let  the sampling set $X=\{x_1,\ldots,x_N\}$ with
\begin{equation*}
x_j= (j,j^2,\ldots,j^{d})/N \mod 1, \quad j=1,\ldots,N.
\end{equation*}
Suppose $N \geq \max\{2q+1,(d-1)^2 (2M-1)^2+1\}$ and $N$ is a prime number. Then OMP (and also BP)
 recovers every $M$-sparse trigonometric polynomial exactly from the determinant sampling  $X$.
\end{thm}

\subsection{Special sparsity patterns}
For the deterministic sampling $X$, in Theorem \ref{th:main1}, an additional assumption is that $N\geq 2q+1=D^{1/d}$. Hence, only when  $q$ is fixed and $d$ is a variable,
we can say that $\max\{2q+1,(M(d-1))^2\}=O(M\log D)^2$.
However, if a prior information about the support of $f$ is known, the restriction might be reduced. For $\Gamma\subset [-q,q]^d\cap \Z^d$, recall that we use $\Pi_\Gamma$  to  denote the space of all trigonometric polynomials in dimension $d$ whose coefficients
 are supported on $\Gamma$.
 We let $\beta_\Gamma$ be the minimum constant so that, for any $k_1,k_2\in \Gamma$ with $k_1\neq k_2$, there exists a non-zero entry of the vector $k_1-k_2$,
 say $k_{1,j}-k_{2,j}$, such that $0<|k_{1,j}-k_{2,j}|\leq \beta_\Gamma$. Then, for $f\in \Pi_\Gamma$, we can replace the condition $N\geq 2q+1$ in Theorem \ref{th:main1}
 by $N\geq \beta_\Gamma+1$.  For example, we suppose that $\Gamma_0$ is a `curve' in $[-q,q]^d$, which is defined by
$$
\Gamma_0:=\left\{\left(m,\left\lfloor\frac{m}{(2q+1)^{1/d}}\right\rfloor,\ldots,\left\lfloor\frac{m}{(2q+1)^{(d-1)/d}}\right\rfloor\right): -q\leq m\leq q\right\}.
$$
Then a simple argument shows that $\beta_{\Gamma_0}\leq (2q+1)^{1/d}$. Hence, OMP can recover $M$-sparse trigonometric polynomials $f\in \Pi_{\Gamma_0}$ exactly from
 determinant sampling $X$ with $N \geq \max\{(2q+1)^{1/d},(d-1)^2 (2M-1)^2+1\}$ and $N$ being a prime number. In particular, if $d\geq \log_2(2q+1)$, the condition
  $N\geq (2q+1)^{1/d}$ is reduced to $N\geq 2$.


\subsection{ Related work } We are not the only ones seeking the deterministic Fourier
sampling. We would especially like to note works by Iwen
\cite{iwen1,iwen2,iwen3} and Bourgain et al \cite{erip}. In \cite{iwen1}, a deterministic Fourier sampling method is considered and a sublinear-time algorithm which recovers  one
dimensional sparse trigonometric polynomials $f$ is presented. However, the sampling in
\cite{iwen1} need a combinatorial structure, which seems  non-trivial to be constructed. Moreover, the coherence of the sampling matrix in
\cite{iwen1} is not small. In recent work
 \cite{iwen2,iwen3}, Iwen presents a binary matrix, say ${\mathcal B}$,
with small coherence. By computing the product of ${\mathcal B}$ and the
discrete Fourier matrix $\Psi$, one can obtain a deterministic  sampling matrix
${\mathcal B}\cdot \Psi$. However, the construction
 of ${\mathcal B}$ requires $O(M\log D)$ large primes and one also needs a fast algorithm for computing the product of ${\mathcal B}$ and $\Psi$.

Our work is in a different direction.  The main concern of our methods is to
present a determinist sampling with an {\em analytic form} so that the
coherence of the sampling matrix is as small as possible and hence the popular
decoder algorithms, such as OMP and BP, can work well for the sampling. In
fact, to construct our deterministic sampling, we only require a prime number
$N$. The coherence of our deterministic sampling matrix almost meets the Welch's
bound. Hence, we also present a nearly optimal harmonic Grassmannian frames,
which may be of independent interest \cite{mu}. We also point out another
difference between our study and that of Iwen. The deterministic sampling given
in \cite{iwen1,iwen2,iwen3} is designed for trigonometric polynomials in a
single variable.   In \cite{iwen3}, the author shows that one can deal with
high dimensional trigonometric polynomials by a dimensionality reduction
technique, but it requires the integer solutions of a linear equation and hence
it is an indirectly method.  In contrast, our sampling is much more convenient
for dealing with high dimensional trigonometric polynomials. The last, but not
 least,  it seems that in practice the algorithms given in \cite{iwen1,iwen3} require
much more samples than OMP and BP when the dimension $d$ is large.

We next compare our study and that of Bourgain et al. In \cite{erip}, for the case with $d=1$, Bourgain et al show a connection between the deterministic Fourier sampling and Tur\'{a}n's problem, and also present many possible approaches to constructing the sampling. Comparing with the approaches given in \cite{erip}, our sampling has the advantage of simplicity. Another important distinction is that our sampling strategy can be interpretable as a sampling strategy for recovering higher dimensional sparse trigonometric polynomials.

Finally, we would like to point out connections with chirp sensing codes.
In \cite{chirp}, a $N\times N^2$ measurement matrix $\Theta$ is designed with chirp sequences forming the columns, i.e.,
$$
(\Theta)_{\ell,k}=\exp(2\pi i r \ell ^2/N) \exp(2\pi i m \ell/N) \, \text{ with } 1\leq \ell \leq N \text{ and } k=Nr+m, 1\leq k \leq N^2.
$$
If we take $d=2$, then ${\mathcal F}_{X}$ is reduced to the chirp sensing matrix $\Theta$. Hence, the measurement matrix $\Theta$ can be considered
 as a special case of our determinant sampling matrix ${\mathcal F}_X$. Noting that the coherence of the matrix ${\mathcal M}(\mathcal F_X/\sqrt{N})$ is not more than $(d-1)/\sqrt{N}$, by an argument   similar to one employed by DeVore in \cite{devore}, we can show that $\mathcal F_X/\sqrt{N}$ satisfies RIP of order $k$ with constant $\delta=(k-1)(d-1)/\sqrt{N}$   (the definition of RIP is given in the next section). For fixed $0<\delta <1$, we can obtain the up bound $k=O(\sqrt{N}/d)$. Using the method of additive combinatorics, in \cite{erip}, Bourgain et al investigate the RIP property of $\Theta$ and show that one can construct  $N\times D$ RIP matrixes of order
    $k$ with $N=o(k^2)$, which implies that a better up bound of $k$. It will be a challenging  problem to extend the result  in \cite{erip} to the matrix ${\mathcal F}_X$.
\section{RIP and StRIP}

Given a $N\times D$ matrix $\Phi$ and any set $T$ of column indices,
we denote by $\Phi_T$ the $n\times |T|$ matrix composed of these
columns. Similarly, for a vector $y\in \C^D$, we use $y_T$ to denote
the vector formed by the entries of $y$ with indices from $T$.
Following Cand\`{e}s and Tao, we say that $\Phi$ satisfies the
Restricted Isometry Property (RIP) of order $k$ and constant $\delta
\in (0,1)$ if
\begin{equation}\label{eq:con}
(1-\delta) \|y_T\|^2 \leq \|\Phi_Ty_T\|^2 \leq (1+\delta) \|y_T\|^2
\end{equation}
holds for all sets $T$ with $|T| \leq k$.
 In fact, (\ref{eq:con}) is equivalent to requiring that the
 Grammian matrices $\Phi_T^\top\Phi_T$ has all of its eigenvalues in $[1-\delta,
 1+\delta]$ for all $T$ with $|T|\leq k$.
 A fundamental question in compressed sensing is the construction of
suitable RIP matrix $\Phi$.
In general,  RIP matrices can be  constructed using random variables such as Gaussian or Bernoulli as  their entries.
However, the construction of the deterministic RIP matrices is a challenging task (see also \cite{devore,erip}).
Given that there are no satisfying deterministic construction of RIP matrices, several authors suggest  an alternative statistical version of  RIP \cite{quad,strip2}:
\begin{defn}  We say that a deterministic matrix
$\Phi$ satisfies statistical restricted isometry property (StRIP) of order $k$ with constant $\delta$ and
probability $1-\epsilon$ if
$$
{\rm Pr}(\left| \|\Phi y\|^2-\|y\|^2\right|\leq \delta\|y\|^2)
\geq 1-\epsilon,
$$
with respect to a uniform distribution of the vectors $y$ among all $k$-sparse
vectors in $\R^D$ of the same norm.
\end{defn}
In \cite{strip1}, the authors derive the StRIP performance bound in
terms of the mutual coherence of the sampling matrix and the
sparsity level of input signal, as summarized by the following
theorem:
\begin{thm} \cite{strip1}\label{th:th32}
Let $\Phi=(\phi_1,\ldots,\phi_D)$ be an $N\times D$ deterministic
 matrix, where each column has unit norm. Assume further
that all row sums of $\Phi$ are equal to zero, and ${\mathcal M}(\Phi)\leq
\alpha_1/\sqrt{N}$, where  $\alpha_1$ is some constant.  Then $\Phi$ is a StRIP of order
$k\leq c_0(\delta)N/\log D$ with constant $\delta$ and probability
$1-1/D$, where $c_0(\delta)=\delta^2/(8\alpha_1^2)$.
\end{thm}
As stated in \cite{quad,strip1,strip2}, there are a large class of
deterministic matrices which satisfy the StRIP. However, the
restriction on the size ($N$ and $D$) is heavy. For example, FZC
Codes \cite{fzc} require that $D=N^2$ and $N$ is a prime number;
Gold/Gold-like Codes \cite{gold} require that $D=(N+1)^2$ and $N$ is
in the form of $2^p-1, p\in \Z$, and so on. We next construct a
deterministic  StRIP  matrix with the large  range size
 based on the method introduced in Section 2.

Suppose $N\in \N$ is a prime number. We suppose that $D\in \N$
can be written in the form of  $D=p_1\cdots p_d$, where $N=p_1\geq
p_2\geq \cdots \geq p_d\geq 2$ and $p_1,\ldots, p_d$ are prime numbers.
For $t=1,\ldots,d$, we set
$$
I_t:=
\begin{cases}
[-(p_t-1)/2,(p_t-1)/2] ,\,\,\, p_t\neq 2,\\
[0,1], \,\,\, p_t=2.
\end{cases}
$$
Recall that
$$
{x}_j= (j,j^2,\ldots,j^{d})/N \mod 1, j=1,\ldots, N,
$$
and denote by $\tilde{\mathcal F}_X$ the $N\times D$ matrix with
entries
$$
(\tilde{\mathcal F}_{X})_{j,k} =\exp{(2\pi ik\cdot
{x}_j)},\quad j=1,\ldots,N,\,\, k\in I_1\times I_2 \times
\cdots \times I_d.
$$
\begin{thm}\label{th:33}
Suppose that $0<\delta<1$ is given. The matrix $\tilde\Phi:=\tilde{\mathcal F}_{X}/\sqrt{N}$ is a StRIP of order $\delta^2
\frac{N}{8(\log D)}\left(\frac{\log (D/N)}{\log N}\right)^2$ with constant $\delta$ and
probability $1-1/D$.
\end{thm}
\begin{proof} We denote the  mutual coherence of $\tilde{\Phi}$ by $\tilde\mu$. Using the same method with the proof of Lemma \ref{le:co},
 we obtain that $\tilde\mu \leq (d-1)/\sqrt{N}$. Note that the all row sums of $\tilde{\mathcal F}_{X}$  is $0$, and also
  that $(d-1)^2 \geq \left(\frac{\log N}{\log (D/N)}\right)^2$. The conclusion can follow from Theorem \ref{th:th32} directly.
\end{proof}
\begin{remark} When $D\gg N$, the Welch's bound (\ref{eq:welch}) approximately equals to $1/\sqrt{N}$. Note that
$$
{\mathcal M}(\tilde{\Phi})\leq (d-1)/\sqrt{N} \leq \log_2D/\sqrt{N}.
$$
The mutual incoherence of $\tilde\Phi$ almost meets the Welch's bound, which implies that $\tilde\Phi$ is a nearly optimal harmonic Grassmannian frame (see \cite{mu}).
\end{remark}
\section{Numerical Experiments}

The purpose of the experiment is the comparison for the random
sampling and the determinant sampling.  Given  the degree of
trigonometric polynomials $q\in \N$ and the number of variables
$d\in \N$, the support set $T$ is drawn from the uniform
distribution over the set of all subsets of $[-q,q]^d$ of size $M$.
The non-zero Fourier coefficients, the real as well as the imaginary
part of $c_j, j\in T$, have the Gaussian distribution with mean zero
and standard deviation one. For the determinant sampling, we use the
method introduced in Section 2 to produce the sampling  points $x_j,j=1,\ldots,N$.
Similar with \cite{trig2},  we choose the random sampling points based on the continuous probability model.

\begin{example}
We take $q=2, d=5, D=3125$ and draw a set $T$ of size $M\in
\{1,2,\ldots,40\}$. The $N=83$ deterministic sampling points are obtained by the method introduced in Section 2, and the
random sampling points are produced by the continuous probability model.  We reconstruct the Fourier
coefficients by OMP  and BP, respectively. In \cite{trig2}, the complexity of OMP is analysised
 in detail. For BP, we use the optimization tools of CVX \cite{cvx}.   We repeat the experiment 100 times for each number $M$ and calculate the success rate.
 Figure 1  and Figure 2 depict the numerical results with the reconstructing algorithm OMP and BP, respectively.  From these, one can observe that
 the performance of deterministic sampling is very similar to those of the random sampling. The numerical experiments imply the promising application of the deterministic sampling.

\begin{figure}[h]
\begin{minipage}[t]{0.51\linewidth}
\centering {
\includegraphics[width=\textwidth]{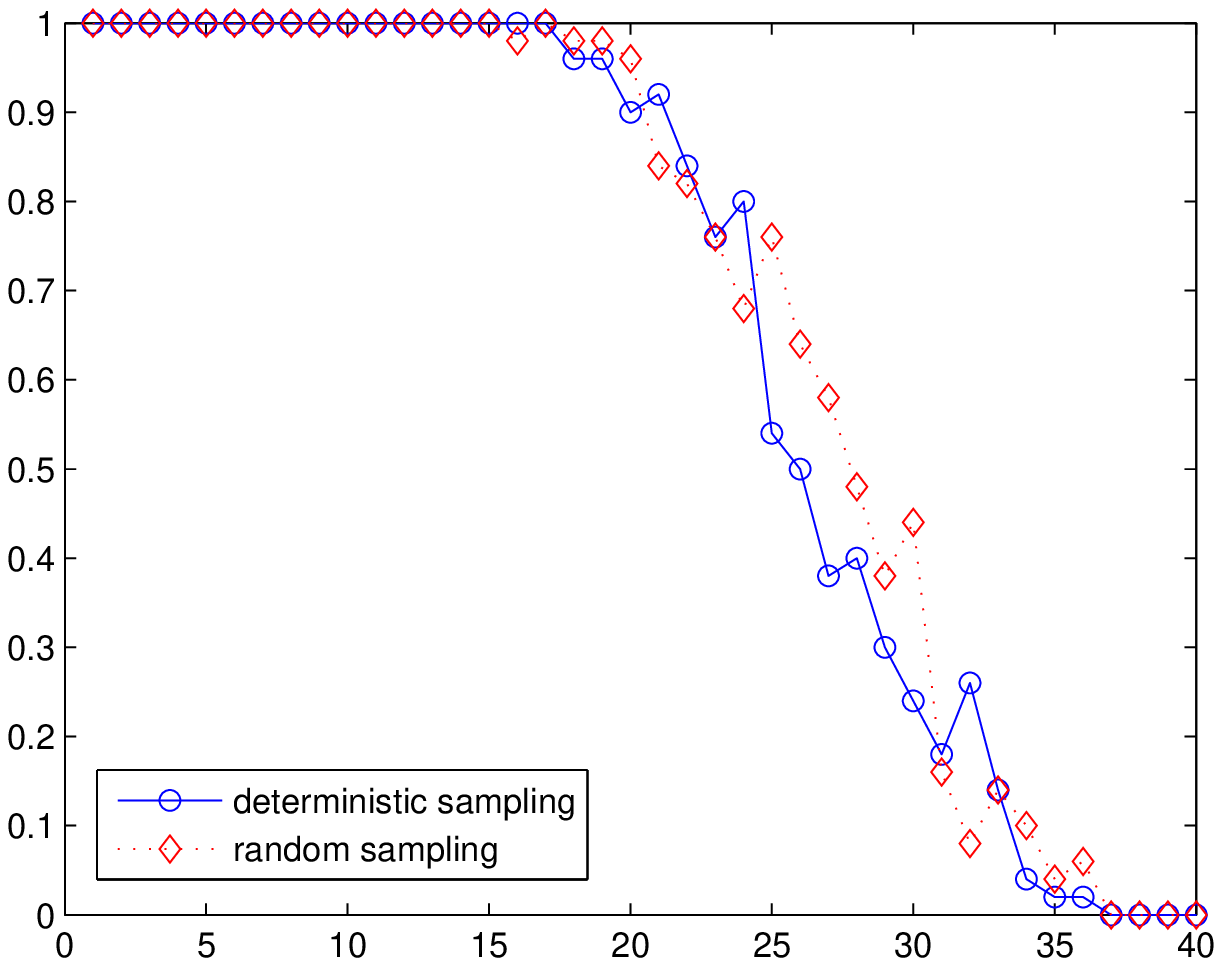}
\caption{ Simulation results using \newline the OMP } }
\end{minipage}
\hfill
\begin{minipage}[t]{0.51\linewidth}
\centering
\includegraphics[width=\textwidth]{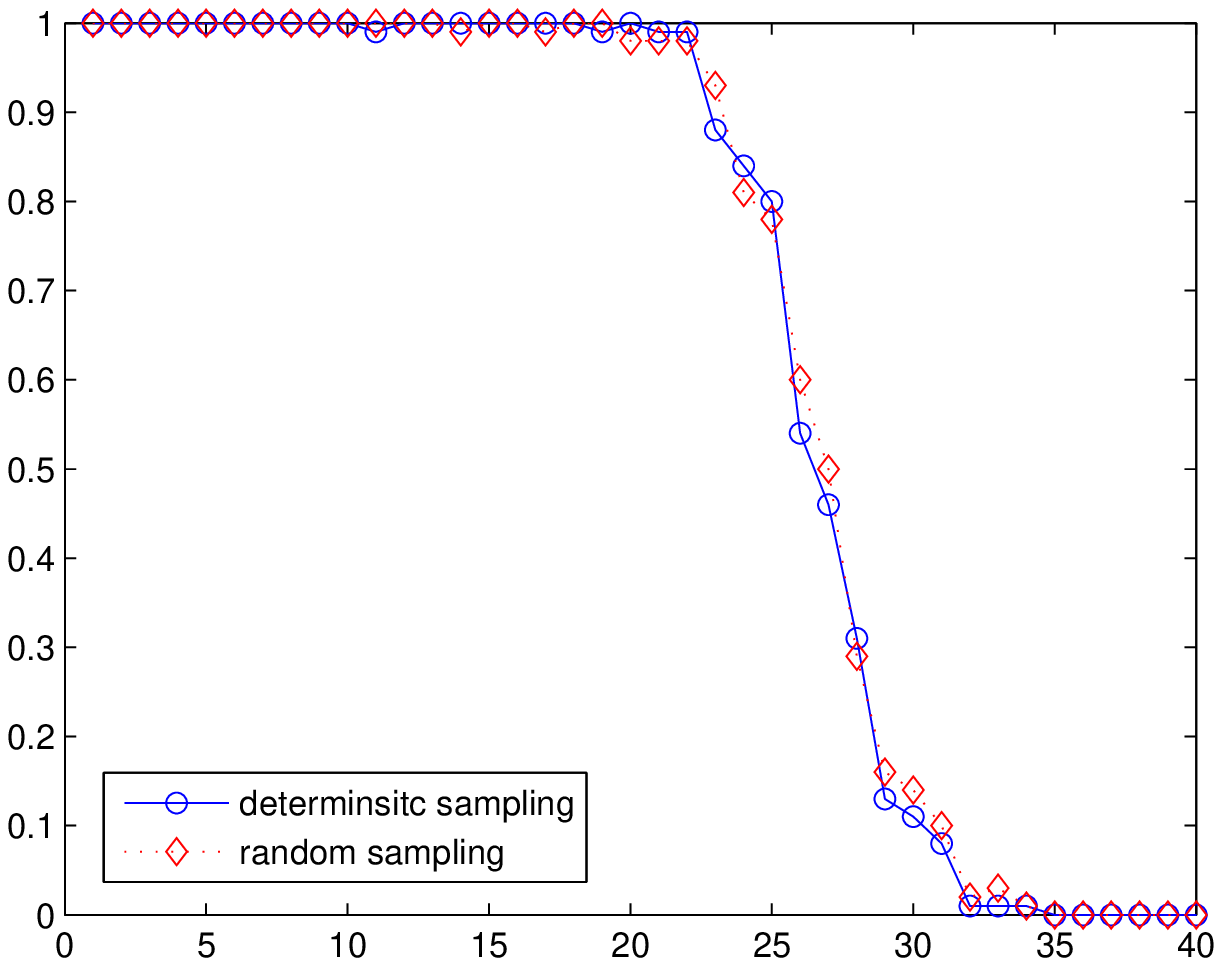}
\caption{Simulation results using\newline the BP\label{your label}}
\end{minipage}
\end{figure}

\end{example}

\begin{example}
As said before, RIP of order $k$ with constant $\delta$ is equivalent to requiring that the
 Grammian matrices $\Phi_T^\top\Phi_T$ has all of its eigenvalues in $[1-\delta,  1+\delta]$ for all $T$ with $|T|\leq k $.
So, the purpose of the second experiment is the comparison for the maximum and minimum eigenvalue statistics of
 Gram matrices $\Phi_T^\top\Phi_T$ of varying sparsity $M:=|T|$ for deterministic sampling matrix and random sampling matrix.
  In order to do this, for every value $M$, sets $T$ are drawn uniformly random over all sets and the statistics are accumulated
  from 50,000 samples. Figure 3 shows the sample means of the maximum and minimum eigenvalues of $\Phi_T^\top\Phi_T$ for  $M\in \{1,\ldots,20\}$.
   \begin{figure}
\begin{center}
\epsfxsize=6cm\epsfbox{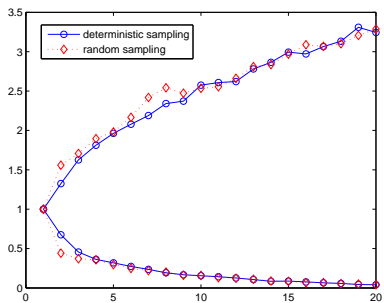} \caption{Eigenvalue statistics of
Gram matrices $\Phi_T^\top\Phi_T$ for deterministic sampling matrixes
and random sampling matrixes }
\end{center}
\end{figure}
\end{example}

\bigskip
{\bf Acknowledgments.} We would like to thank the referees for valuable comments on this paper.

\bigskip \medskip

\noindent {\bf Authors' addresses:}

\medskip

\noindent Zhiqiang Xu,
LSEC, Inst.~Comp.~Math., Academy of Mathematics and Systems Science,
 Chinese Academy of Sciences, Beijing, 100091, China,
  {\tt Email: xuzq@lsec.cc.ac.cn}

\end{document}